\documentclass[12pt,reqno]{amsart}
\usepackage{amsfonts,amsmath,amssymb,amsthm}
\usepackage{braket,caption,color,dsfont,enumerate,esint,float,graphicx,hyperref,mathdots,mathtools}
\usepackage{perpage,pdfpages,pgfplots,thmtools}
\usepackage{tikz,tikz-3dplot}
\usepackage[margin=2.2cm]{geometry}

\allowdisplaybreaks 
\MakePerPage{footnote}

\newcommand\Ab{\mathbb{A}}
\newcommand\Ac{\mathcal{A}}

\newcommand\C{\mathbb{C}}
\newcommand\Card{\mathrm{Card}}
\newcommand\Cc{\mathcal{C}}
\newcommand\Char{\mathds{1}}

\newcommand\E{\mathbb{E}}

\newcommand\Fc{\mathcal{F}}

\newcommand\Gb{\mathbb{G}}

\newcommand\Lip{\mathrm{Lip}}

\newcommand\N{\mathbb{N}}

\newcommand\R{\mathbb{R}}

\newcommand\Var{\mathrm{Var}}
\newcommand\Vol{\mathrm{Vol}}

\newcommand\ve{\varepsilon}

\newcommand\Z{\mathbb{Z}}

\theoremstyle{plain}
\newtheorem{thm}{\textbf{Theorem}}[section]

\newtheorem{cor}[thm]{\textbf{Corollary}}
\newtheorem{defn}[thm]{\textbf{Definition}}

\newtheorem{prop}[thm]{\textbf{Proposition}}

\theoremstyle{remark}

\newtheorem*{rmk}{\textbf{Remark}}

\numberwithin{equation}{section}

\title{Fractal uncertainty principle for random Cantor sets}

\author{Xiaolong Han}
\email{xiaolong.han@csun.edu}

\author{Pouria Salekani}
\email{pouria.salekani.474@my.csun.edu}
\address{Department of Mathematics, California State University, Northridge, CA 91330, USA}



\thanks{} 

\begin{document}
\maketitle

\begin{abstract}
We continue our investigation of the fractal uncertainty principle (FUP) for random fractal sets. In the prequel \cite{EH}, we considered the Cantor sets in the discrete setting with alphabets randomly chosen from a base of digits so the dimension $\delta\in(0,\frac23)$. We proved that, with overwhelming probability, the FUP with an exponent $\ge\frac12-\frac34\delta-\ve$ holds for these discrete Cantor sets with random alphabets.

In this sequel, we construct random Cantor sets with dimension $\delta\in(0,\frac23)$ in $\R$ via a different random procedure from the one in \cite{EH}. We prove that, with overwhelming probability, the FUP with an exponent $\ge\frac12-\frac34\delta-\ve$ holds. The proof follows from establishing a Fourier decay estimate of the corresponding random Cantor measures, which is in turn based on a concentration of measure phenomenon in an appropriate probability space for the random Cantor sets.
\end{abstract}

\section{Introduction}\label{sec:intro}
The uncertainty principle is a fundamental concept asserting that no function can be simultaneously localized in both position and frequency. This principle carries profound implications throughout mathematics and physics. Recently, Dyatlov-Zahl \cite{DZ} introduced the \textit{fractal uncertainty principle} (FUP), which provides a quantitative limitation on how a function and its Fourier transform may both concentrate near fractal sets. Since its introduction, the FUP has become a central theme in Fourier analysis and its deep connections with other fields, including chaotic dynamics, spectral theory, and scattering theory.

The FUP has been established in a variety of one-dimensional settings through several distinct approaches: additive combinatorics (Dyatlov-Zahl \cite{DZ}, Cladek-Tao \cite{CT}), Fourier decay of fractal measures (\cite{BD1}), dynamical techniques (Dyatlov-Jin \cite{DJ2}, Backus-Leng-Tao \cite{BLT}), and Beurling-Malliavin theory (Bourgain-Dyatlov \cite{BD2}, Cohen \cite{C2}). Higher-dimensional extensions have also been obtained; see for instance Han-Schlag \cite{HS} and Cohen \cite{C1, C2}.

To illustrate one of its central applications, consider an open quantum system. Resonant states oscillate while allowing energy to escape to infinity. The \textit{spectral gap} governs the slowest rate of decay, thereby determining the asymptotic rate of energy dissipation with significant consequences for the stability of nonlinear wave equations. Resonant states associated with the slowest decay typically concentrate near trapped sets, that is, points that remain confined under the underlying classical dynamics. Because of the chaotic nature of such systems, these trapped sets exhibit a fractal structure. The FUP imposes a quantitative restriction on the localization of resonant states, leading directly to estimates for the spectral gap.

Two principal examples of chaotic systems illustrating this mechanism are: (1) open baker’s maps, a canonical example of a discrete chaotic dynamical system, where the trapped sets are Cantor sets. In this setting, the FUP leads directly to a spectral gap for the corresponding quantum baker’s maps \cite{DJ1, EH}; and (2) the geodesic flow on convex co-compact hyperbolic surfaces, a prominent example of a continuous chaotic dynamical system, where the trapped sets correspond to Schottky limit sets. The FUP established for these Schottky limit sets has contributed directly to recent breakthroughs on the spectral gap for the Laplacian in this geometry \cite{BD1, BD2, DZ}. For a comprehensive overview of the FUP and its applications to open systems, as well as to eigenfunction estimates in closed systems, we refer the reader to Dyatlov’s survey article \cite{Dy}.

We now state the formal framework of the FUP. Let $0<h\le1$ denote the semiclassical parameter and define the semiclassical Fourier transform
$$
\Fc_hu(\xi)=\frac{1}{\sqrt{2\pi h}}\int_{\R}e^{-\frac{ix\cdot \xi}{h}}u(x)\,dx\quad\text{for }u\in C_0^\infty(\R).
$$
For $h=1$, this coincides with the standard Fourier transform $\Fc:=\Fc_1$. The FUP concerns estimates for the operator norm
\begin{equation}\label{eq:FUP}
\left\|\Char_X\Fc_h\Char_Y\right\|_{L^2(\R)\to L^2(\R)},
\end{equation}
where the $h$-dependent sets $X=X(h)$ and $Y=Y(h)\subset\R$ have fractal-type geometry characterized by $\delta$-regularity, following Dyatlov \cite[Definition 2.2]{Dy}.

\begin{defn}[$\delta$-regular sets]\label{defn:reg}
Let $0\le\delta\le1$, $R\ge1$, and $0\le\alpha_{\mathrm{min}}\le\alpha_{\mathrm{max}}\le\infty$. A non-empty closed set $X\subset\R$ is called $\delta$-regular with constant $R$ on scales $\alpha_{\mathrm{min}}$ to $\alpha_{\mathrm{max}}$ if there exists a locally finite measure $\nu_X$ supported on $X$ such that, for every interval $I$ centered at a point of $X$ with $\alpha_{\mathrm{min}}\le|I|\le\alpha_{\mathrm{max}}$, we have
$$R^{-1}|I|^\delta\le\nu_X(I)\le R|I|^\delta.$$
Here, $|E|$ denotes the Lebesgue measure of a measurable set $E\subset\R$.
\end{defn}

Since $\Fc_h$ is unitary on $L^2(\R)$, we always have \eqref{eq:FUP}$=O(1)$. We say that $X$ and $Y$ satisfy the FUP with exponent $\beta\ge0$ if $\eqref{eq:FUP}=O(h^\beta)$ as $h\to0$. The main goal in studying the FUP for $\delta$-regular sets is to establish the existence of an exponent $\beta>0$ and to determine the sharp exponent for which the FUP holds. The \textit{sharp exponent}, denoted by $\beta^s(X,Y)$, is the largest $\beta$ such that $\eqref{eq:FUP}=O(h^\beta)$ holds for all $0<h<h_0$ with some $h_0>0$. It typically depends on the parameters $\delta$ and $R$ in Definition \ref{defn:reg} associated with the sets $X$ and $Y$.

First, if $X$ and $Y$ are $\delta$-regular on scales $h$ to $1$ with $\delta\in[0,1]$, then the FUP \eqref{eq:FUP} holds with exponent
\begin{equation}\label{eq:volbd}
\beta_\Vol=\max\left\{\frac12-\delta,0\right\}.
\end{equation}
This quantity \eqref{eq:volbd} is sometimes referred to as ``the volume bound'', as it depends only on the volumes of $X$ and $Y$, namely $|X|,|Y|\le Ch^{1-\delta}$; see Bourgain-Dyatlov \cite[Lemma 2.9]{BD2}.

Improvements over the volume bound \eqref{eq:volbd} have been obtained in various settings \cite{BD1, BD2, BLT, C1, C2, CT, DJ1, DJ2, HS, JZ}. The sets considered in these works are deterministic, and the improvement in the exponent is generally small (and either implicitly or explicitly dependent on $\delta$ and $R$ in Definition \ref{defn:reg}). In particular, the following theorem includes explicit estimates on such improvements, which are exponentially small when $0<\delta\le\frac12$ (Dyatlov-Jin \cite[Theorem 1]{DJ2}) and super-exponentially small when $\frac12<\delta<1$ (Jin-Zhang \cite[Theorem 1.2]{JZ}).

\begin{thm}\label{thm:FUPdeter}
Let $0<\delta<1$ and $R\ge1$. Suppose that $X,Y$ be $\delta$-regular with constant $R$ on scales $h$ to $1$. Then the FUP \eqref{eq:FUP} holds with an exponent $\beta$ such that
$$\beta-\beta_\Vol\ge\begin{cases}
(5R)^{-\frac{40}{\delta(1-\delta)}} & \text{if }0<\delta\le\frac12,\\
\exp\left[-\exp\left(C\left(R\delta^{-1}(1-\delta)^{-1}\right)^{C(1-\delta)^{-2}}\right)\right] & \text{if }\frac12<\delta<1.
\end{cases}$$
Here, $C>0$ is an absolute constant.
\end{thm}

\begin{rmk}
The FUP depends on the framework of Fourier transform $\Fc_h$. In particular, a non-trivial function $u$ and its Fourier transform $\Fc_h u$ cannot be both compactly supported. However, in the framework of Walsh-Fourier transform, such phenomenon is allowed. In terms of the FUP, there are certain Cantor sets of dimension $\frac12\le\delta<1$ for which the FUP \eqref{eq:FUP} does \textit{not} hold with exponents greater than $\beta_\Vol$ (which is $0$ in this case). This is due to Demeter \cite{De}.
\end{rmk}

In an ongoing project, we investigate the FUP \eqref{eq:FUP} when the sets $X,Y$ are constructed via certain random procedures, and in these random settings, we aim to prove the FUP with more favorable exponents than the ones in Theorem \ref{thm:FUPdeter} for the deterministic cases. In the paper \cite{EH} by Eswarathasan and the first-named author, we considered the FUP for the random Cantor sets in the discrete setting. That is, let $M,A\in\N$ such that $M\ge3$ and $A=M^\delta$ with $0<\delta<1$. Then the alphabets of cardinality $A$ from the digits $\{0,...,M-1\}$ form a probability space 
$$\Ab(M,A)=\left\{\Ac\subset\{0,...,M-1\}:\Card(\Ac)=A\right\},$$
which is equipped with the uniform counting measure $\mu_1$. Here, $\Card(\Ac)$ is the cardinality of $\Ac$.  

In \cite{EH}, an alphabet is chosen at random from $\Ab(M,A)$; then the discrete Cantor sets are constructed using this alphabet (in each iteration step). Let $\delta\in(0,\frac23)$. We proved that, with overwhelming probability (that is, except on a subset of $\Ab(M,A)$ with exponentially small measure depending on $M,\ve$), the Cantor sets with random alphabets satisfy the FUP with an exponent $\beta\ge\frac12-\frac34\delta-\ve$. It is a significant improvement over the volume bound \eqref{eq:volbd} for this random ensemble.

In this paper, we consider the FUP for random Cantor sets in the continuous setting of $\R$. We use a different random procedure with the one in \cite{EH} for the construction. See Section \ref{sec:comp} for the different random ensembles of Cantor sets, as well as the comparison of approaches and results between the discrete and continuous settings.

Each Cantor set in $\R$ can be built via an iteration process. Let $\Ab(M,A)$ be as above. Inductively define 
$$B_1=\frac1M\Ac,\quad\text{where }\Ac\in\Ab(M,A),$$
and for $j\ge2$,
$$B_j=\bigcup_{b\in B_{j-1}}\left\{b+\frac{a}{M^j}:a\in\Ac(b)\right\},$$
where $\Ac(b)\in\Ab(M,A)$ with $b\in B_{j-1}$ are independent and identical distributed (iid) random variables. In particular, $\Card(B_j)=A^j$ for all $j\in\N$.

Write
\begin{equation}\label{eq:Cc}
\Cc_j=\bigcup_{b\in B_j}\left[b,b+\frac{1}{M^j}\right],\quad\text{and}\quad\Cc=\bigcap_{j=1}^\infty\Cc_j\subset[0,1].
\end{equation}
Hence, $\Cc_j$ is a union of $A^j$ closed intervals, each of which has Lebesgue volume of $M^{-j}$. Define also the Borel measure $\nu_j$ whose density function is given by
$$\rho_j(x)=\frac{M^j}{A^j}\Char_{\Cc_j}(x)\quad\text{for }x\in\R.$$
Finally, the weak limit $\nu$ of $\nu_j$ as $j\to\infty$ defines the Cantor measure which is supported on $\Cc$. 

We next set up the appropriate probability space for the random Cantor set $\Cc$ and the random Cantor measure $\nu$. In the first iteration, the probability space is given by $\Ab(M,A)$; in the $j$-th iteration for $j\ge2$, since $\Card(B_{j-1})=A^{j-1}$, the probability space is given by $\Ab(M,A)^{A^{j-1}}$ equipped with the uniform counting measure $\mu_j$. Write
\begin{equation}\label{eq:Ab}
\Ab^\infty:=\prod_{j=1}^\infty\Ab(M,A)^{A^{j-1}}\text{ equipped with the probability measure }\mu^\infty:=\prod_{j=1}^\infty\mu_j.
\end{equation}
Then $\Cc$ is a random Cantor set and $\nu$ is a random Cantor measure in $\R$ with respect to the probability space $\Ab^\infty$. 

It is obvious that each Cantor set $\Cc$ constructed in this way has Hausdorff dimension
$$\delta=\frac{\log A}{\log M}\in(0,1).$$
Moreover, each Cantor set $\Cc$ is $\delta$-regular with constant $R=2$ on scales $0$ to $1$, for which one can simply choose the Cantor measure $\nu$ in Definition \ref{defn:reg}, see Dyatlov \cite[Example 2.6]{Dy} for the example of the middle-third Cantor set (that is, $M=3$ and $\Ac=\{0,2\}$). 

Consider the $h$-neighborhoods $\Cc(h)$ of $\Cc$. Then $\Cc(h)$ is $\delta$-regular with constant $R=8$ on scales $h$ to $1$, see Bourgain-Dyatlov \cite[Lemma 2.3]{BD2}. Therefore, for \textit{each} Cantor set $\Cc$, the FUP \eqref{eq:FUP} holds for $X=Y=\Cc(h)$ with an exponent $\beta$ explicitly given in Theorem \ref{thm:FUPdeter}.

However, in the case when $0<\delta<\frac23$, we prove that, with overwhelming probability (that is, except on a subset of $\Ab^\infty$ with exponentially small measure), the FUP holds for the random Cantor sets with much better exponents than the ones in Theorem \ref{thm:FUPdeter}: 
\begin{thm}[FUP for random Cantor sets]\label{thm:FUPL}
Let $0<\delta<\frac23$. Suppose that $M,A\in\N$ such that $M\ge 2^\frac4\delta$ and $A=M^\delta$. Then for $0<\ve<\frac\delta2$, there exists $\Gb\subset\Ab^\infty$ with
$$\mu^\infty\left(\Ab^\infty\setminus\Gb\right)\le C_1e^{-\frac{M^\ve}{C}}$$
such that each Cantor set $\Cc$ from $\Gb$ satisfies that
$$\left\|\Char_{\Cc(h)}\Fc_h\Char_{\Cc(h)}\right\|_{L^2(\R)\to L^2(\R)}\le Ch^{\frac12-\frac34\delta-\ve}\quad\text{for all }0<h<M^{-8}.$$
Here, $C>0$ is an absolute constant and $C_1=C_1(\ve)>0$ depends on $\ve$.
\end{thm}

We also prove an FUP for the random Cantor measures $\nu$:
\begin{thm}[FUP for random Cantor measures]\label{thm:FUPC}
Under the same conditions as Theorem \ref{thm:FUPL}, the corresponding random Cantor measure $\nu$ satisfies that
$$\left\|\int_{\R}e^{-\frac{ix\cdot \xi}{h}}u(x)\,d\nu(x)\right\|_{L^2_\nu(\R)}\le Ch^{\frac\delta4-\ve}\|u\|_{L^2_\nu(\R)}\quad\text{for all }u\in L^2_\nu(\R)\text{ and }0<h<M^{-8}.$$
\end{thm}

The proofs of the FUP in Theorem \ref{thm:FUPL} and \ref{thm:FUPC} are based on the following Fourier decay estimate of the random Cantor measures:
\begin{thm}[Fourier decay of random Cantor measures]\label{thm:FD}
Let $0<\delta<1$. Suppose that $M,A\in\N$ such that $M\ge 2^\frac4\delta$ and $A=M^\delta$. Then for $0<\ve<\min\{\frac\delta2,\frac13\}$, there exists $\Gb\subset\Ab^\infty$ with
$$\mu^\infty\left(\Ab^\infty\setminus\Gb\right)\le C_1e^{-\frac{M^\ve}{C}}$$
such that each Cantor measure $\nu$ from $\Gb$ satisfies that
\begin{equation}\label{eq:FD}
|\Fc\nu(\xi)|\le C|\xi|^{-\frac{\delta-\ve}{2}}\quad\text{for all }|\xi|\ge M^4.
\end{equation}
Here, $C>0$ is an absolute constant and $C_1=C_1(\ve)>0$ depends on $\ve$.
\end{thm}

\begin{rmk}[Fourier dimension and Salem sets]
The Fourier dimension $\dim_\mathrm{F}E$ of a Borel set $E\subset\R$ is the largest number $s$ such that there is a finite Borel measure $\nu$ supported on $E$ which satisfies that $\Fc\nu(\xi)=O(|\xi|^{-\frac s2})$. We always have that $\dim_\mathrm{F}E\le\dim_\mathrm{H}E$, the Hausdorff dimension of $E$, because the strongest Fourier decay estimate for a measure supported on $E$ is $O(|\xi|^{-\frac{\dim_\mathrm{H}E}{2}})$. We say that $E$ is a Salem set if $\dim_\mathrm{F}F=\dim_\mathrm{H}E$. Salem \cite{Sa} constructed the first example of such sets, which is random in nature and is Cantor-like. See Mattila \cite[Chapter 12]{M} for more details.

The random Cantor measures in Theorem \ref{thm:FD} satisfy (almost) the strongest possible Fourier decay estimates, valid for the full range $\delta \in (0,1)$. However, these estimates yield the fractal uncertainty principle in Theorems \ref{thm:FUPL} and \ref{thm:FUPC} only in the more restricted range $\delta \in (0,\tfrac{2}{3})$. Our construction of random Cantor sets is largely inspired by Shmerkin \cite[Theorem 2.1]{Sh}, who established (almost) optimal Fourier decay estimates for a broader class of random Cantor sets. In that setting, though, the decay estimate for $\Fc\nu(\xi)$ was proved only at integer frequencies, which is insufficient for our application, and the constant $C$ in \eqref{eq:FD} depends on the specific Cantor set.

Finally, we remark that the Fourier decay of random fractal measures is closely connected to arithmetic progressions and Fourier restriction problems; see \L{}aba-Pramanik \cite{LP} and \L{}aba-Wang \cite{LW} for further background and recent developments.
\end{rmk}

\subsection{Random ensembles in the discrete setting and the continuous setting}\label{sec:comp}
Continue with our notations of $\Ab(M,A)$ with $M,A\in\N$ and $A=M^\delta$, $0<\delta<1$. We discuss the FUP for different ensembles of Cantor sets. The boundary points of the Cantor set in the continuous setting of $\R$ at each iteration $j\in\N$ naturally define the Cantor set in the discrete setting of 
$$\Z_N:=\frac1N\Z/N\Z\quad\text{with }N=M^j.$$
This connection allows us to compare the approaches and results on the FUP in the discrete setting and in the continuous setting. (We shall remark that for direct comparison between the two settings, the discrete set $\Z_N$ here differs with the one in \cite{DJ1, EH} by a scaling factor of $N$.) 

\textbf{Ensemble I}. Let $\Ac\in\Ab(M,A)$ be randomly chosen. For $j\in\N$, define the discrete Cantor set of order $j$ as
$$B_j=\left\{\sum_{l=1}^j\frac{a_l}{M^l}:a_l\in\Ac\text{ for }l=1,...,j\right\},$$
for which the probability space is $\Ab(M,A)$.

\textbf{Ensemble II}. For $j\in\N$, let $\Ac_k\in\Ab(M,A)$ for $l=1,...,j$ be iid random variables. Define the discrete Cantor set of order $j$ as
$$B_j=\left\{\sum_{l=1}^j\frac{a_l}{M^l}:a_l\in\Ac_l\text{ for }l=1,...,j\right\},$$
for which the probability space is $\prod_{l=1}^j\Ab(M,A)$.

\textbf{Ensemble III}. Let $B_1=\frac1M\Ac$ with $\Ac\in\Ab(M,A)$. For $j\ge2$, let $\Ac(b)\in\Ab(M,A)$ for $b\in B_{j-1}$ be iid random variables, and define the discrete Cantor set of order $j$ as
$$B_j=\bigcup_{b\in B_{j-1}}\left\{b+\frac{a}{M^j}:a\in\Ac(b)\right\},$$
for which the probability space is $\prod_{k=1}^j\Ab(M,A)^{A^{k-1}}$.

In each of the ensembles above, the Cantor sets $\Cc$ in $\R$ are defined by \eqref{eq:Cc}. The following three figures demonstrate (the initial iterations of) examples of the Cantor sets with $M=3$ and $A=2$.

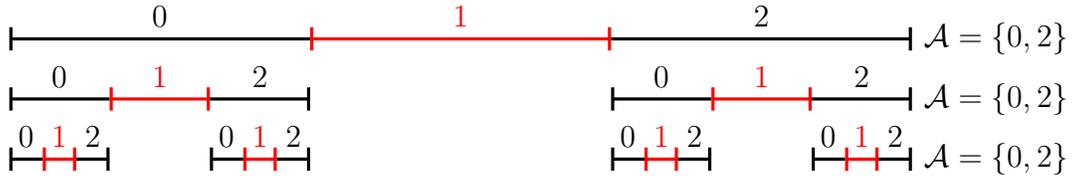
\begin{figure}[H]
\centering
\begin{tikzpicture}
\draw[black,very thick,|-] (0,0) -- (4,0) node[midway,above] {$0$};
\draw[red,very thick,|-|] (4,0) -- (8,0) node[midway,above] {$1$};
\draw[black,very thick,-|] (8,0) -- (12,0) node[midway,above] {$2$} node[right] {$\Ac=\{0,2\}$}; 
\draw[black,very thick,|-] (0,-.8) -- (4/3,-.8) node[midway,above] {$0$};
\draw[red,very thick,|-|] (4/3,-.8) -- (8/3,-.8) node[midway,above] {$1$};
\draw[black,very thick,-|] (8/3,-.8) -- (4,-.8) node[midway,above] {$2$}; 
\draw[black,very thick,|-] (8,-.8) -- (28/3,-.8) node[midway,above] {$0$};
\draw[red,very thick,|-|] (28/3,-.8) -- (32/3,-.8) node[midway,above] {$1$};
\draw[black,very thick,-|] (32/3,-.8) -- (12,-.8) node[midway,above] {$2$} node[right] {$\Ac=\{0,2\}$}; 
\draw[black,very thick,|-] (0,-1.6) -- (4/9,-1.6) node[midway,above] {$0$};
\draw[red,very thick,|-|] (4/9,-1.6) -- (8/9,-1.6) node[midway,above] {$1$};
\draw[black,very thick,-|] (8/9,-1.6) -- (4/3,-1.6) node[midway,above] {$2$};
\draw[black,very thick,|-] (8/3,-1.6) -- (28/9,-1.6) node[midway,above] {$0$}; 
\draw[red,very thick,|-|] (28/9,-1.6) -- (32/9,-1.6) node[midway,above] {$1$}; 
\draw[black,very thick,-|] (32/9,-1.6) -- (4,-1.6) node[midway,above] {$2$}; 
\draw[black,very thick,|-] (8,-1.6) -- (76/9,-1.6) node[midway,above] {$0$};
\draw[red,very thick,|-|] (76/9,-1.6) -- (80/9,-1.6) node[midway,above] {$1$};
\draw[black,very thick,-|] (80/9,-1.6) -- (84/9,-1.6) node[midway,above] {$2$};
\draw[black,very thick,|-] (96/9,-1.6) -- (100/9,-1.6) node[midway,above] {$0$}; 
\draw[red,very thick,|-|] (100/9,-1.6) -- (104/9,-1.6) node[midway,above] {$1$}; 
\draw[black,very thick,-|] (104/9,-1.6) -- (12,-1.6) node[midway,above] {$2$} node[right] {$\Ac=\{0,2\}$}; 
\end{tikzpicture}
\caption{The initial three iterations and the alphabet used for a Cantor set in Ensemble I. (The intervals colored red are removed in the iteration process.)}
\end{figure}

\begin{figure}[H]
\centering
\begin{tikzpicture}
\draw[black,very thick,|-|] (0,0) -- (4,0) node[midway,above] {$0$};
\draw[black,very thick,-] (4,0) -- (8,0) node[midway,above] {$1$};
\draw[red,very thick,|-|] (8,0) -- (12,0) node[midway,above] {$2$};
\draw[] (12,0) node[right] {$\Ac_1=\{0,1\}$}; 
\draw[black,very thick,|-] (0,-.8) -- (4/3,-.8) node[midway,above] {$0$};
\draw[red,very thick,|-|] (4/3,-.8) -- (8/3,-.8) node[midway,above] {$1$};
\draw[black,very thick,-|] (8/3,-.8) -- (4,-.8) node[midway,above] {$2$}; 
\draw[black,very thick,-] (4,-.8) -- (16/3,-.8) node[midway,above] {$0$};
\draw[red,very thick,|-|] (16/3,-.8) -- (20/3,-.8) node[midway,above] {$1$};
\draw[black,very thick,-|] (20/3,-.8) -- (8,-.8) node[midway,above] {$2$};
\draw[] (12,-.8) node[right] {$\Ac_2=\{0,2\}$}; 
\draw[red,very thick,|-|] (0,-1.6) -- (4/9,-1.6) node[midway,above] {$0$};
\draw[black,very thick,-|] (4/9,-1.6) -- (8/9,-1.6) node[midway,above] {$1$};
\draw[black,very thick,-|] (8/9,-1.6) -- (4/3,-1.6) node[midway,above] {$2$};
\draw[red,very thick,|-|] (8/3,-1.6) -- (28/9,-1.6) node[midway,above] {$0$}; 
\draw[black,very thick,-|] (28/9,-1.6) -- (32/9,-1.6) node[midway,above] {$1$}; 
\draw[black,very thick,-] (32/9,-1.6) -- (4,-1.6) node[midway,above] {$2$}; 
\draw[red,very thick,|-|] (4,-1.6) -- (40/9,-1.6) node[midway,above] {$0$};
\draw[black,very thick,-|] (40/9,-1.6) -- (44/9,-1.6) node[midway,above] {$1$};
\draw[black,very thick,-|] (44/9,-1.6) -- (48/9,-1.6) node[midway,above] {$2$};
\draw[red,very thick,|-|] (60/9,-1.6) -- (64/9,-1.6) node[midway,above] {$0$}; 
\draw[black,very thick,-|] (64/9,-1.6) -- (68/9,-1.6) node[midway,above] {$1$}; 
\draw[black,very thick,-|] (68/9,-1.6) -- (8,-1.6) node[midway,above] {$2$};
\draw[] (12,-1.6) node[right] {$\Ac_3=\{1,2\}$}; 
\end{tikzpicture}
\caption{The initial three iterations and the alphabets used for a Cantor set in Ensemble II.}
\end{figure}
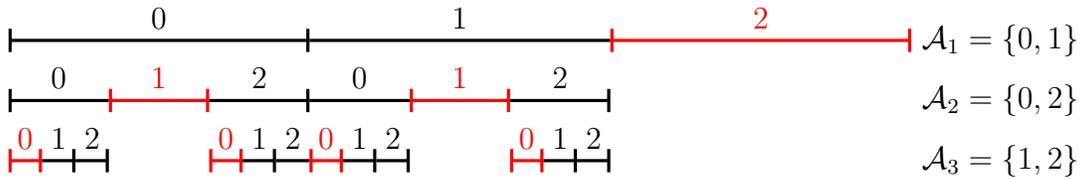

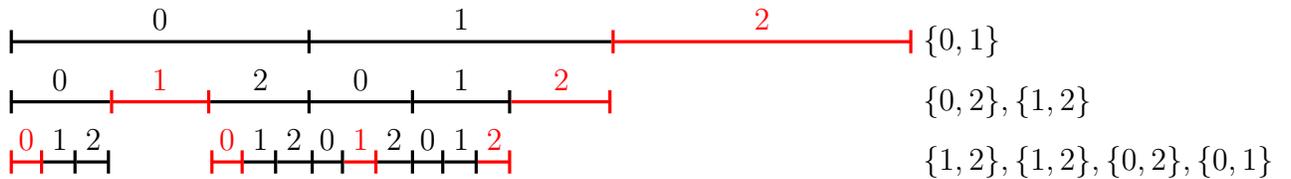
\begin{figure}[H]
\centering
\begin{tikzpicture}
\draw[black,very thick,|-|] (0,0) -- (4,0) node[midway,above] {$0$};
\draw[black,very thick,-] (4,0) -- (8,0) node[midway,above] {$1$};
\draw[red,very thick,|-|] (8,0) -- (12,0) node[midway,above] {$2$};
\draw[] (12,0) node[right] {$\{0,1\}$}; 
\draw[black,very thick,|-] (0,-.8) -- (4/3,-.8) node[midway,above] {$0$};
\draw[red,very thick,|-|] (4/3,-.8) -- (8/3,-.8) node[midway,above] {$1$};
\draw[black,very thick,-|] (8/3,-.8) -- (4,-.8) node[midway,above] {$2$}; 
\draw[black,very thick,-] (4,-.8) -- (16/3,-.8) node[midway,above] {$0$};
\draw[black,very thick,|-|] (16/3,-.8) -- (20/3,-.8) node[midway,above] {$1$};
\draw[red,very thick,-|] (20/3,-.8) -- (8,-.8) node[midway,above] {$2$};
\draw[] (12,-.8) node[right] {$\{0,2\},\{1,2\}$}; 
\draw[red,very thick,|-|] (0,-1.6) -- (4/9,-1.6) node[midway,above] {$0$};
\draw[black,very thick,-|] (4/9,-1.6) -- (8/9,-1.6) node[midway,above] {$1$};
\draw[black,very thick,-|] (8/9,-1.6) -- (4/3,-1.6) node[midway,above] {$2$};
\draw[red,very thick,|-|] (8/3,-1.6) -- (28/9,-1.6) node[midway,above] {$0$}; 
\draw[black,very thick,-|] (28/9,-1.6) -- (32/9,-1.6) node[midway,above] {$1$}; 
\draw[black,very thick,-] (32/9,-1.6) -- (4,-1.6) node[midway,above] {$2$}; 
\draw[black,very thick,|-|] (4,-1.6) -- (40/9,-1.6) node[midway,above] {$0$};
\draw[red,very thick,-|] (40/9,-1.6) -- (44/9,-1.6) node[midway,above] {$1$};
\draw[black,very thick,-] (44/9,-1.6) -- (48/9,-1.6) node[midway,above] {$2$};
\draw[black,very thick,|-|] (48/9,-1.6) -- (52/9,-1.6) node[midway,above] {$0$}; 
\draw[black,very thick,-|] (52/9,-1.6) -- (56/9,-1.6) node[midway,above] {$1$}; 
\draw[red,very thick,-|] (56/9,-1.6) -- (60/9,-1.6) node[midway,above] {$2$};
\draw[] (12,-1.6) node[right] {$\{1,2\},\{1,2\},\{0,2\},\{0,1\}$}; 
\end{tikzpicture}
\caption{The initial three iterations and the alphabets used for a Cantor set in Ensemble III.}
\end{figure}

\begin{rmk}[The FUP for Cantor sets in the discrete setting of $\Z_N$]
Let $\Fc_N$ be the discrete Fourier transform which is unitary on $l^2(\Z_N)$. For each $j\in\N$, the discrete Cantor set $B_j\in\Z_N$. The FUP for Cantor sets in the discrete setting is concerned with the estimate of the form
\begin{equation}\label{eq:FUPd}
\left\|\Char_{B_j}\Fc_N\Char_{B_j}\right\|_{l^2(\Z_N)\to l^2(\Z_N)}=O\left(N^{-\beta}\right)\quad\text{as }j\to\infty,
\end{equation}
where $N^{-1}=M^{-j}\to0$ plays the role of the semiclassical parameter. 

For the deterministic Cantor sets in Ensemble I, Dyatlov-Jin \cite{DJ1} introduced a FUP theory and proved the FUP \eqref{eq:FUPd} with exponent $\beta=\beta(M,\Ac)>\beta_\Vol$ for all $M,\Ac$. They also extensively studied the sharp exponent $\beta_s(M,\Ac)$ in the FUP for different $M,\Ac$. On one hand, they found examples of $M,\Ac$ for which the FUP holds with the exponent $\frac{1-\delta}{2}$, which is the best possible one. On the other hand, they also found examples of $M,\Ac$ for which the sharp exponent $\beta_s(M,\Ac)=\beta_\Vol+o_M(1)$; see Hu \cite{H} and Kangabire \cite{K} for more recent constructions of Cantor sets for such examples.

For the random Cantor sets in Ensemble I, Eswarathasan-Han \cite{EH} introduced a probabilistic approach to the FUP and proved that for $0<\delta<\frac23$, $\beta_s(M,\Ac)\ge\frac12-\frac34\delta-\ve$ with overwhelming probability. The proof is similar to the one used in this paper; that is, it is based on establishing a Fourier decay estimate of the corresponding discrete Cantor measure $\tilde\nu_j=\sum_{b\in B_j}\delta_b$. However, the proof is simplified because of the submultiplicativity property in the discrete setting: If $j=j_1+j_2$, then
$$\left\|\Char_{B_j}\Fc_N\Char_{B_j}\right\|_{l^2(\Z_N)\to l^2(\Z_N)}\le\left\|\Char_{B_{j_1}}\Fc_{N_1}\Char_{B_{j_1}}\right\|_{l^2(\Z_{N_1})\to l^2(\Z_{N_1})}\left\|\Char_{B_{j_2}}\Fc_{N_2}\Char_{B_{j_2}}\right\|_{l^2(\Z_{N_2})\to l^2(\Z_{N_2})}$$
where $N=N_1N_2$ with $N_1=M^{j_1}$ and $N_2=M^{j_2}$. Hence, the estimate of the FUP at the initial iteration $j=1$ implies one for all iterations $j\in\N$. Indeed, the Fourier transform of the discrete Cantor measure at the first iteration,
$$\Fc\tilde\nu_1(m),\quad\text{where }\tilde\mu_1=\sum_{b\in B_1}\delta_b=\sum_{a\in\Ac}\delta_{\frac aM},$$ 
has decay for $m\in\Z_M\setminus\{0\}$ for randomly chosen $\Ac\in\Ab(M,A)$, which is sufficient to imply the FUP in the discrete setting.

For the deterministic Cantor sets in Ensembles II and III, one can prove an FUP with exponents in Theorem \ref{thm:FUPdeter}, following an approach of Dyatlov-Jin \cite{DJ2} to reduce the FUP for the discrete Fourier transform $\Fc_N$ to one for $\Fc_h$ with respect to the discrete measures $\tilde\nu_j$. Moreover, Dyatlov-Jin \cite{DJ1} still supplies examples of Cantor sets for which the FUP holds with the best possible exponent (by simply choosing the same alphabet at each iteration).

However, much less is known for the random Cantor sets in Ensembles II and III than Ensemble I. In particular, the submultiplicativity property is not necessarily true, since different alphabets can be used in different iterations. So one has to consider higher iterations when proving the FUP via the probabilistic approach. This is open.
\end{rmk}

\begin{rmk}[The FUP for Cantor sets in the continuous setting of $\R$]
All Cantor sets $\Cc$ in each ensemble above are $\delta$-regular with an absolute constant $R$ on scales $0$ to $1$ in Definition \ref{defn:reg}. Their $h$-neighborhood $\Cc(h)$ are $\delta$-regular on scales $h$ to $1$. As a consequence, the $h$-neighborhoods $\Cc(h)$ satisfy the FUP in Theorem \ref{thm:FUPdeter}. 

Our main results in Theorems \ref{thm:FUPL} and \ref{thm:FUPC} are concerned with the random Cantor sets in Ensemble III, which state that if $0<\delta<\frac23$, then the sharp constant $\beta_s\ge\frac12-\frac34\delta-\ve$ in the FUP with overwhelming probability. 

However, for the random Cantor sets in Ensembles I and II, much less is known. In particular, the Cantor measures in these ensembles do not necessarily have Fourier decay; that is, the Fourier dimension of some Cantor sets can be zero, see again Mattila \cite[Chapter 12]{M}. The probabilistic approach to the FUP in these ensembles is open.

Moreover, for the deterministic case in all the ensembles above, no examples of Cantor sets are known to satisfy the FUP with the best possible exponent $\frac{1-\delta}{2}$ in the continuous setting of $\R$ to the authors' knowledge.
\end{rmk}

\subsection{Organization of the paper}
In Section \ref{sec:prob}, we prepare the probabilisitic tools from concentration of measure theory for our proof of the main theorems. In Section \ref{sec:FD}, we prove the Fourier decay of the random Cantor measures. In Section \ref{sec:FUP}, we use the Fourier decay to establish the FUP for the random Cantor measures and random Cantor sets.

\subsection*{Acknowledgements}
We would like to thank the referees for their careful reading of the manuscript and for the valuable suggestions that helped improve the presentation of the paper.

\section{Probabilistic estimates}\label{sec:prob}
Recall that $M,A\in\N$ with $M\ge3$ and $A=M^\delta$ with $0<\delta<1$, and that $\Ab(M,A)$ (equipped with the uniform counting probability measure $\mu_1$) is the probability space of alphabets of cardinality $A$ from the digits $\{0,...,M-1\}$. 

In this section, we establish the probabilistic estimates, which are used to prove the Fourier decay estimate of random Cantor measures, see Section \ref{sec:FD}. This involves the estimate of
$$\frac{1}{\Card(S)}\sum_{s\in S}f(s),$$
where $S$ is a discrete probability space and $s\in S$ are independent and identically distributed random variables. In particular, the following function \eqref{eq:B} appears naturally in such an estimate.

Suppose that $B$ is a finite set and $\Ac(b)\in\Ab(M,A)$ for $b\in B$ are independent and identical random variables. For $N\in\N$ and $\eta\in\R\setminus\{0\}$, define
\begin{equation}\label{eq:B}
\frac{1}{\Card(B)}\sum_{b\in B}e^{iN\eta b}F_\eta\left(\Ac(b)\right),
\end{equation}
in which
\begin{equation}\label{eq:F}
F_\eta(\Ac)=\frac1A\sum_{a\in\Ac}e^{-i\eta a}-\frac1M\sum_{a=0}^{M-1}e^{-i\eta a}.
\end{equation}
Hence, \eqref{eq:B} is a random variable with respect to the probability space $\Ab(M,A)^{\Card(B)}$ equipped with the uniform counting measure $\mu$. The main estimate in this section is
\begin{thm}\label{thm:P}
Let $N\in\N$ and $\eta\in\R\setminus\{0\}$. Then
$$\mu\left(\left|\frac{1}{\Card(B)}\sum_{b\in B}e^{iN\eta b}F_\eta\left(\Ac(b)\right)\right|\ge t\right)\le2\exp\left(-\frac{\Card(B)At^2}{64+\frac43At}\right)\quad\text{for all }t>0.$$
\end{thm} 

We estimate $F_\eta(\Ac)$ for $\Ac\in\Ab(M,A)$ and first observe that 
\begin{equation}\label{eq:Fmax}
\left|F_\eta(\Ac)\right|\le2\quad\text{for all }\Ac\in\Ab(M,A).
\end{equation}
We next show that the random variable $F_\eta(\Ac)$ for $\Ac\in\Ab(M,A)$ is concentrated near its expectation at an exponential rate, which therefore satisfies a much better estimate than the above trivial bound with overwhelming probability. It is a consequence of the concentration of measure theory in the metric space $\Ab(M,A)$ established by Eswarathasan-Han \cite[Section 2]{EH}. 

Set the metric in $\Ab(M,A)$ by
\begin{eqnarray}
d(\Ac_1,\Ac_2)&=&\Card\left(\Ac_1\triangle\Ac_2\right)\nonumber\\
&=&\Card(\Ac_1\setminus\Ac_2)+\Card(\Ac_2\setminus\Ac_1)\quad\text{for }\Ac_1,\Ac_2\in\Ab(M,A).\label{eq:dist}
\end{eqnarray}
Here, $\Ac_1\triangle\Ac_2$ denotes the symmetric difference. For a function $F:\Ab(M,A)\to\C$, its Lipschitz norm is defined by
$$\|F\|_\Lip=\max_{\Ac_1,\Ac_2\in\Ab(M,A),\Ac_1\ne\Ac_2}\frac{|F(\Ac_1)-F(\Ac_2)|}{d(\Ac_1,\Ac_2)}.$$
Under this setup, we have Eswarathasan-Han \cite[Theorem 2.1]{EH}:
\begin{thm}[Concentration of measure in the space of alphabets]\label{thm:comA}
Let $F:\Ab(M,A)\to\C$ and $t>0$. Then
$$\mu_1\left(\left\{\Ac\in\Ab(M,A):|F(\Ac)-\E[F]|\ge t\right\}\right)\le2\exp\left(-\frac{t^2}{16A\|F\|^2_\Lip}\right),$$
in which $\E[F]$ is the expectation of $F$ with respect to $\mu_1$.
\end{thm}
To apply Theorem \ref{thm:comA} to $F_\eta$ in \eqref{eq:F}, we compute that
\begin{eqnarray*}
\E\left[F_\eta\right]&=&\frac{1}{\Card\left(\Ab(M,A)\right)}\sum_{\Ac\in\Ab(M,A)}F_\eta(\Ac)\\
&=&{M\choose A}^{-1}\sum_{\Ac\in\Ab(M,A)}\left(\frac1A\sum_{a\in\Ac}e^{-i\eta a}-\frac1M\sum_{a=0}^{M-1}e^{-i\eta a}\right)\\
&=&\frac1A{M\choose A}^{-1}\sum_{a=0}^{M-1}\sum_{\Ac\in\Ab(M,A),\Ac\ni a}e^{-i\eta a}-\frac1M\sum_{a=0}^{M-1}e^{-i\eta a}\\
&=&\frac1A{M\choose A}^{-1}{M-1\choose A-1}\sum_{a=0}^{M-1}e^{-i\eta a}-\frac1M\sum_{a=0}^{M-1}e^{-i\eta a}\\
&=&\frac1M\sum_{a=0}^{M-1}e^{-i\eta a}-\frac1M\sum_{a=0}^{M-1}e^{-i\eta a}\\
&=&0.
\end{eqnarray*}
Here, we used the fact that for any fixed digit $a=0,...,M-1$, the number of alphabets $\Ac$ in $\Ab(M,A)$ which contain $a$ is ${M-1\choose A-1}$. 

To estimate the Lipschitz norm of $F_\eta$ in \eqref{eq:F}, let $\Ac_1,\Ac_2\in\Ab(M,A)$. Then
\begin{eqnarray*}
\left|F_\eta\left(\Ac_1\right)-F_\eta\left(\Ac_2\right)\right|&=&\frac1A\left|\sum_{a\in\Ac_1}e^{-i\eta a}-\sum_{a\in\Ac_2}e^{-i\eta a}\right|\\
&\le&\frac1A\cdot\Card(\Ac_1\triangle\Ac_2)\\
&=&\frac1A\cdot d(\Ac_1,\Ac_2),
\end{eqnarray*}
in the view of the metric \eqref{eq:dist}. Hence, $\|F\|_\Lip\le\frac1A$. By Theorem \ref{thm:comA}, we therefore have the following proposition.
\begin{prop}\label{prop:comA}
Let $\eta\in\R\setminus\{0\}$ and $t>0$. Then
$$\mu_1\left(\left\{\Ac\in\Ab(M,A):\left|F_\eta(\Ac)\right|\ge t\right\}\right)\le2\exp\left(-\frac{At^2}{16}\right).$$
\end{prop}

An immediate consequence is an estimate on the variance $\Var[F_\eta]$ of $F_\eta$:
\begin{cor}\label{cor:var}
Let $\eta\in\R\setminus\{0\}$. Then
$$\Var\left[F_\eta\right]\le\frac{32}{A}.$$
\end{cor}
\begin{proof}
Since $\E[F_\eta]=0$, compute that
\begin{eqnarray*}
\Var\left[F_\eta\right]&=&\E\left[\left|F_\eta\right|^2\right]\\
&=&2\int_0^\infty t\cdot\mu_1\left(\left\{\Ac\in\Ab(M,A):\left|F_\eta(\Ac)\right|\ge t\right\}\right)\,dt\\
&\le&4\int_0^\infty te^{-\frac{At^2}{16}}\,dt\\
&=&\frac{32}{A}.
\end{eqnarray*}
\end{proof}

We now prove Theorem \ref{thm:P} and need the following version of Bernstein inequality from Bennett \cite{B}.
\begin{thm}[Bernstein inequality]
Let $X_1,...,X_n$ be independent and identical random variables with expectation zero. Suppose that for some $C>0$, $|X_j|\le C$ for all $j=1,...,n$. Then for $t>0$,
$$P\left(\left|\frac1n\sum_{j=1}^nX_j\right|\ge t\right)\le2\exp\left(-\frac{nt^2}{\frac2n\sum_{j=1}^n\Var\left[X_j\right]+\frac23Ct}\right).$$
\end{thm}

\begin{proof}[Proof of Theorem \ref{thm:P}]
We apply Bernstein inequality with $n=\Card(B)$ and
$$X(b)=e^{iN\eta b}F_\eta\left(\Ac(b)\right),$$
for which $X(b)\le2$ from \eqref{eq:Fmax} for all $b\in B$. Moreover, 
$$\E[X(b)]=0\quad\text{and}\quad\Var[X(b)]\le\frac{32}{A}.$$
Hence,
$$\mu\left(\left|\frac{1}{\Card(B)}\sum_{b\in B}e^{i\eta b}F_\eta\left(\Ac(b)\right)\right|\ge t\right)\le2\exp\left(-\frac{\Card(B)t^2}{\frac{64}{A}+\frac43t}\right)=2\exp\left(-\frac{\Card(B)At^2}{64+\frac43At}\right).$$
\end{proof}

\section{Fourier decay of random Cantor measures}\label{sec:FD}
We derive the Fourier transform of the Cantor measure $\nu$ supported on the Cantor set $\Cc$, see \eqref{eq:Cc}. Let $\xi\in\R\setminus\{0\}$. Then 
$$\Fc\nu(\xi)=\lim_{j\to\infty}\Fc\nu_j(\xi)=\frac{1}{\sqrt{2\pi}}\lim_{j\to\infty}\int_\R e^{-ix\xi}\rho_j(x)\,dx.$$
Here, the density function $\rho_j$ of the measure $\nu_j$ assigns a measure of $\frac{M^j}{A^j}$ to each of the $A^j$ intervals which defines $\Cc_j$. Hence,
\begin{eqnarray*}
\Fc\nu_j(\xi)&=&\frac{1}{\sqrt{2\pi}}\sum_{b\in B_j}\frac{M^j}{A^j}\int_b^{b+\frac{1}{M^j}}e^{-ix\xi}\,dx\\
&=&\frac{M^j}{\sqrt{2\pi}A^j}\sum_{b\in B_j}\frac{e^{-i\left(b+\frac{1}{M^j}\right)\xi}-e^{-ib\xi}}{-i\xi}\\
&=&\frac{i\left(e^{-i\xi/M^j}-1\right)}{\sqrt{2\pi}\xi/M^j}\cdot\left(\frac{1}{A^j}\sum_{b\in B_j}e^{-i\xi b}\right).
\end{eqnarray*}
Since $\Card(B_j)=A^j$,
$$\left|\Fc\nu_j(\xi)\right|\le\left|\frac{i\left(e^{-i\xi/M^j}-1\right)}{\sqrt{2\pi}\xi/M^j}\cdot\left(\frac{1}{A^j}\sum_{b\in B_j}e^{-i\xi b}\right)\right|\le\left|\frac{i\left(e^{-i\xi/M^j}-1\right)}{\sqrt{2\pi}\xi/M^j}\right|\le C\cdot\min\left\{1,\frac{M^j}{|\xi|}\right\}.$$
in which $C>0$ is an absolute constant; that is, it is independent of $M,A,\xi,j$. Here, we used the fact that
$$\frac{\left(e^{-i\eta}-1\right)}{\eta}=O(1)\text{ as }\eta\to0,\quad\text{and}\quad\frac{\left(e^{-i\eta}-1\right)}{\eta}=O\left(|\eta|^{-1}\right)\text{ as }|\eta|\to\infty.$$
This estimate of $\Fc\nu_j$ is not strong enough to derive a Fourier decay of $\nu$. To do so, we next compare the Fourier transforms of $\Cc_j$ for consecutive $j$'s, which requires a different way of computing $\Fc\nu_j$ as above. Let $j\ge2$. Then
\begin{eqnarray*}
\Fc\nu_{j-1}(\xi)&=&\frac{1}{\sqrt{2\pi}}\sum_{b\in B_{j-1}}\frac{M^{j-1}}{A^{j-1}}\int_b^{b+\frac{1}{M^{j-1}}}e^{-ix\xi}\,dx\\
&=&\frac{1}{\sqrt{2\pi}}\sum_{b\in B_{j-1}}\frac{M^{j-1}}{A^{j-1}}\sum_{a=0}^{M-1}\int_{b+\frac{a}{M^j}}^{b+\frac{a+1}{M^j}}e^{-ix\xi}\,dx\\
&=&\frac{1}{\sqrt{2\pi}}\sum_{b\in B_{j-1}}\frac{M^{j-1}}{A^{j-1}}\sum_{a=0}^{M-1}\frac{e^{-i\left(b+\frac{a+1}{M^j}\right)\xi}-e^{-i\left(b+\frac{a}{M^j}\right)\xi}}{-i\xi}\\
&=&\frac{i\left(e^{-i\xi/M^j}-1\right)}{\sqrt{2\pi}\xi/M^j}\cdot\left(\frac{1}{A^{j-1}}\sum_{b\in B_{j-1}}e^{-i\xi b}\cdot\frac1M\sum_{a=0}^{M-1}e^{-i\xi a/M^j}\right).
\end{eqnarray*}
Similarly, 
\begin{eqnarray*}
\Fc\nu_j(\xi)&=&\frac{1}{\sqrt{2\pi}}\sum_{b\in B_j}\frac{M^j}{A^j}\int_b^{b+\frac{1}{M^j}}e^{-ix\xi}\,dx\\
&=&\frac{1}{\sqrt{2\pi}}\sum_{b\in B_{j-1}}\frac{M^j}{A^j}\sum_{a\in\Ac(b)}\int_{b+\frac{a}{M^j}}^{b+\frac{a+1}{M^j}}e^{-ix\xi}\,dx\\
&=&\frac{1}{\sqrt{2\pi}}\sum_{b\in B_{j-1}}\frac{M^j}{A^j}\sum_{a\in\Ac(b)}\frac{e^{-i\left(b+\frac{a+1}{M^j}\right)\xi}-e^{-i\left(b+\frac{a}{M^j}\right)\xi}}{-i\xi}\\
&=&\frac{i\left(e^{-i\xi/M^j}-1\right)}{\sqrt{2\pi}\xi/M^j}\cdot\left(\frac{1}{A^{j-1}}\sum_{b\in B_{j-1}}e^{-i\xi b}\cdot\frac1A\sum_{a\in\Ac(b)}e^{-i\xi a/M^j}\right).
\end{eqnarray*}
Denoting $\eta=\xi/M^j$, we have that
\begin{eqnarray*}
&&\Fc\nu_j(\xi)-\Fc\nu_{j-1}(\xi)\\
&=&\frac{i\left(e^{-i\xi/M^j}-1\right)}{\sqrt{2\pi}\xi/M^j}\cdot\left[\frac{1}{A^{j-1}}\sum_{b\in B_{j-1}}e^{-i\xi b}\left(\frac1A\sum_{a\in\Ac(b)}e^{-i\xi a/M^j}-\frac1M\sum_{a=0}^{M-1}e^{-i\xi a/M^j}\right)\right]\\
&=&\frac{i\left(e^{-i\eta}-1\right)}{\sqrt{2\pi}\eta}\cdot\left[\frac{1}{A^{j-1}}\sum_{b\in B_{j-1}}e^{-iM^j\eta b}\left(\frac1A\sum_{a\in\Ac(b)}e^{-i\eta a}-\frac1M\sum_{a=0}^{M-1}e^{-i\eta a}\right)\right]\\
&=&\frac{i\left(e^{-i\eta}-1\right)}{\sqrt{2\pi}\eta}\cdot\left[\frac{1}{A^{j-1}}\sum_{b\in B_{j-1}}e^{-iM^j\eta b}F_\eta\left(\Ac(b)\right)\right]\\
&:=&G(\eta).
\end{eqnarray*}
We use Theorem \ref{thm:P} to estimate $G(\eta)$. Firstly, compute that
\begin{eqnarray*}
G'(\eta)&=&\frac{\eta e^{-i\eta}-i\left(e^{-i\eta}-1\right)}{\sqrt{2\pi}\eta^2}\cdot\left[\frac{1}{A^{j-1}}\sum_{b\in B_{j-1}}e^{-iM^j\eta b}F_\eta\left(\Ac(b)\right)\right]\\
&&+\frac{i\left(e^{-i\eta}-1\right)}{\sqrt{2\pi}\eta}\cdot\left[\frac{1}{A^{j-1}}\sum_{b\in B_{j-1}}e^{-iM^j\eta b}\left(-iM^jbF_\eta\left(\Ac(b)\right)+\partial_\eta F_\eta\left(\Ac(b)\right)\right)\right].
\end{eqnarray*}
Since $a\in\{0,1,...,M-1\}$,
$$\left|\partial_\eta F_\eta\left(\Ac(b)\right)\right|=\left|\frac1A\sum_{a\in\Ac(b)}\left(-iae^{-i\eta a}\right)-\frac1M\sum_{a=0}^{M-1}\left(-iae^{-i\eta a}\right)\right|\le2M.$$
Since $|F_\eta(\Ac)|\le2$ by \eqref{eq:Fmax} and $B_{j-1}\subset[0,1]$ with $\Card(B_{j-1})=A^{j-1}$, we have that
\begin{eqnarray*}
\left|G'(\eta)\right|&\le&\left|\frac{\eta e^{-i\eta}-i\left(e^{-i\eta}-1\right)}{\sqrt{2\pi}\eta^2}\right|\cdot\left|\frac{1}{A^{j-1}}\sum_{b\in B_{j-1}}e^{-iM^j\eta b}F_\eta\left(\Ac(b)\right)\right|\\
&&+\left|\frac{i\left(e^{-i\eta}-1\right)}{\sqrt{2\pi}\eta}\right|\cdot\left|\frac{1}{A^{j-1}}\sum_{b\in B_{j-1}}e^{-iM^j\eta b}\left(-iM^jbF_\eta\left(\Ac(b)\right)+\partial_\eta F_\eta\left(\Ac(b)\right)\right)\right|\\
&\le&C\cdot\min\left\{1,|\eta|^{-1}\right\}+C\min\left\{1,|\eta|^{-1}\right\}\cdot\left(M^j+M\right)\\
&\le&C\cdot\min\left\{1,|\eta|^{-1}\right\}\cdot M^j,
\end{eqnarray*}
in which $C>0$ is an absolute constant. Here, we used the fact that
$$\frac{\eta e^{-i\eta}-i\left(e^{-i\eta}-1\right)}{\eta^2}=O(1)\text{ as }\eta\to0,\quad\text{and}\quad\frac{\eta e^{-i\eta}-i\left(e^{-i\eta}-1\right)}{\eta^2}=O\left(|\eta|^{-1}\right)\text{ as }|\eta|\to\infty.$$
Next we move on to estimating the following term in $G(\eta)$:
$$\frac{1}{A^{j-1}}\sum_{b\in B_{j-1}}e^{-iM^j\eta b}F_\eta\left(\Ac(b)\right)=\frac{1}{A^{j-1}}\sum_{b\in B_{j-1}}e^{-iM^j\eta b}\left(\frac1A\sum_{a\in\Ac(b)}e^{-i\eta a}-\frac1M\sum_{a=0}^{M-1}e^{-i\eta a}\right),$$
which has a period of $2\pi$. Thus, it suffices to consider this term only in the case when $\eta\in(0,2\pi]$. Let 
$$0<L\le A^{\frac j2-1}.$$
Set $K$ as the smallest integer such that
$$l:=\frac{2\pi}{K}\le\frac{\sqrt{2\pi}L}{A^{\frac j2}M^j}.$$
Then
$$\left\lfloor\frac{\sqrt{2\pi}A^{\frac j2}M^j}{L}\right\rfloor\le K\le\left\lfloor\frac{\sqrt{2\pi}A^{\frac j2}M^j}{L}\right\rfloor+1.$$
Divide $(0,2\pi]$ into sub-intervals of equal length $l$. Denote the boundary points of these sub-intervals (except $0$) by
$$\eta_k=kl,\quad\text{in which }k=1,2,...,K.$$
We are now ready to apply Theorem \ref{thm:P} to $\eta_k$. Here, the probability space $\Ab(M,A)^{A^{j-1}}$ for $\Cc_j$ is equipped with the probability measure $\mu_j$. Take $t=LA^{-\frac j2}$. Then there is $\Omega_k\subset\Ab(M,A)^{A^{j-1}}$ (depending on $\eta_k$) with
$$\mu_j\left(\Ab(M,A)^{A^{j-1}}\setminus\Omega_k\right)\le2\exp\left(-\frac{A^{j-1}\cdot A\left(LA^{-\frac j2}\right)^2}{64+\frac43A\left(LA^{-\frac j2}\right)}\right)\le2\exp\left(-\frac{L^2}{66}\right)$$
such that for all $\Cc_j$ in $\Omega_k$,
$$\left|\frac{1}{A^{j-1}}\sum_{b\in B_{j-1}}e^{-iM^j\eta_k b}F_{\eta_k}\left(\Ac(b)\right)\right|\le LA^{-\frac j2}.$$
Write
\begin{equation}\label{eq:Gj}
\Gb_j=\bigcap_{k=1}^K\Omega_k.
\end{equation}
Then for some absolute constant $C>0$,
\begin{eqnarray*}
\mu_j\left(\Ab(M,A)^{A^{j-1}}\setminus\Gb_j\right)&\le&\sum_{k=1}^K\mu_j\left(\Ab(M,A)^{A^{j-1}}\setminus\Omega_k\right)\\
&\le&2K\exp\left(-\frac{L^2}{66}\right)\\
&\le&C\cdot\frac{A^{\frac j2}M^j}{L}\cdot\exp\left(-\frac{L^2}{C}\right),
\end{eqnarray*}
and for all $\Cc_j$ in $\Gb_j$,
$$\left|\frac{1}{A^{j-1}}\sum_{b\in B_{j-1}}e^{-iM^j\eta_k b}F_{\eta_k}\left(\Ac(b)\right)\right|\le LA^{-\frac j2}\quad\text{for all }k=1,...,K.$$
Hence, under the same conditions,
$$\left|G\left(\eta_k\right)\right|=\left|\frac{i\left(e^{-i\eta_k}-1\right)}{\sqrt{2\pi}\eta_k}\right|\cdot\left|\frac{1}{A^{j-1}}\sum_{b\in B_{j-1}}e^{-iM^j\eta b}F_{\eta_k}\left(\Ac(b)\right)\right|\le C\cdot\min\left\{1,\left|\eta_k\right|^{-1}\right\}\cdot LA^{-\frac j2}.$$
To pass the estimate for $\eta_k$, $k=1,...,K$, to the one for all $\eta\in(0,2\pi]$, we use the mean value theorem of $G(\eta)$. That is, for each $\eta\in(0,2\pi]$, there is $\eta_k$ such that $\eta_{k-1}=\eta_k-l<\eta\le\eta_k$. Thus,
\begin{eqnarray*}
\left|G\left(\eta_k\right)-G(\eta)\right|&\le&\sup_{\zeta\in\left(\eta,\eta_k\right)}|G'(\zeta)|\cdot l\\
&\le&C\cdot\min\left\{1,|\eta|^{-1}\right\}\cdot M^j\cdot\frac{\sqrt{2\pi}L}{A^{\frac j2}M^j}\\
&\le&C\cdot\min\left\{1,|\eta|^{-1}\right\}\cdot LA^{-\frac j2}.
\end{eqnarray*}
Finally,
$$|G(\eta)|\le\left|G\left(\eta_k\right)-G(\eta)\right|+\left|G\left(\eta_k\right)\right|\le C\cdot\min\left\{1,|\eta|^{-1}\right\}\cdot LA^{-\frac j2}.$$
Recall that $\eta=\xi/M^j$. We have established that
\begin{prop}
Let $j\ge2$ and $0<L\le A^{\frac j2-1}$. Then there is $\Gb_j\subset\Ab(M,A)^{A^{j-1}}$ with
$$\mu_j\left(\Ab(M,A)^{A^{j-1}}\setminus\Gb_j\right)\le C\cdot\frac{A^{\frac j2}M^j}{L}\cdot\exp\left(-\frac{L^2}{C}\right)$$
such that for all $\Cc_j$ in $\Gb_j$ (that is, $\{\Ac(b),b\in B_{j-1}\}$ are from $\Gb_j$),
\begin{equation}\label{eq:Fnu}
\left|\Fc\nu_j(\xi)-\Fc\nu_{j-1}(\xi)\right|\le C\cdot\min\left\{1,\frac{M^j}{|\xi|}\right\}\cdot LA^{-\frac j2}\quad\text{for all }\xi\in\R\setminus\{0\}.
\end{equation}
Here, $C>0$ is an absolute constant.
\end{prop}
Let
$$0<\ve\le\frac13.$$
For $j\ge3$, set 
$$L=L_j=M^{\frac{j\ve}{2}}.$$
Since $\frac{j\ve}{2}\le\frac j2-1$ for all $j\ge3$, we apply the proposition above. Then there is $\Gb_j\in\Ab(M,A)^{A^{j-1}}$ with
$$\mu_j\left(\Ab(M,A)^{A^{j-1}}\setminus\Gb_j\right)\le C\cdot\frac{A^{\frac j2}M^j}{L_j}\cdot\exp\left(-\frac{L_j^2}{C}\right)$$
such that for all $\Cc_j$ in $\Gb_j$, \eqref{eq:Fnu} holds. 

Set $\Gb_1=\Ab(M,A)$, $\Gb_2=\Ab(M,A)^A$, and $\Gb_j$ for $j\ge3$ be given in \eqref{eq:Gj}. Let
$$\Gb=\prod_{j=1}^\infty\Gb_j\subset\Ab^\infty=\prod_{j=1}^\infty\Ab(M,A)^{A^{j-1}}.$$
Recall that $A=M^\delta$ with $0<\delta<1$. Hence,
\begin{eqnarray*}
\mu^\infty\left(\Ab^\infty\setminus\Gb\right)&\le&\sum_{j=1}^\infty\mu_j\left(\Ab(M,A)^{A^{j-1}}\setminus\Gb_j\right)\\
&\le&\sum_{j=3}^\infty C\cdot\frac{A^{\frac j2}M^j}{L_j}\cdot\exp\left(-\frac{L_j^2}{C}\right)\\
&\le&\sum_{j=3}^\infty C\cdot M^{\frac{3j}{2}}\cdot\exp\left(-\frac{M^{j\ve}}{C}\right)\\
&\le&C\exp\left(-\frac{M^\ve}{C}\right)\sum_{j=3}^\infty M^{\frac{3j}{2}}\cdot\exp\left(-\frac{M^{(j-1)\ve}}{C}\right)\\
&\le&C_1\exp\left(-\frac{M^\ve}{C}\right).
\end{eqnarray*}
Here, $C>0$ is an absolute constant and $C_1=C_1(\ve)>0$ depends on $\ve$.

Now suppose that $\Cc$ is in $\Gb$; that is, $\Cc_j$ are in $\Gb_j$ for all $j\in\N$, see \eqref{eq:Ab}. Then
\begin{eqnarray*}
|\Fc\nu(\xi)|&\le&\left|\Fc\nu_2(\xi)\right|+\sum_{j=3}^\infty\left|\Fc\nu_j(\xi)-\Fc\nu_{j-1}(\xi)\right|\\
&\le&C\cdot\min\left\{1,\frac{M^2}{|\xi|}\right\}+\sum_{j=3}^\infty C\cdot\min\left\{1,\frac{M^j}{|\xi|}\right\}\cdot L_jA^{-\frac j2}.
\end{eqnarray*}
Divide the summation into the cases when $j<J$ and $j\ge J$, in which
$$J=\left\lfloor\frac{\log|\xi|}{\log M}\right\rfloor.$$
Assume further that 
$$0<\ve\le\min\left\{\frac{\delta}{2},\frac13\right\}.$$
Then
\begin{eqnarray*}
&&\sum_{j=3}^\infty C\cdot\min\left\{1,\frac{M^j}{|\xi|}\right\}\cdot L_jA^{-\frac j2}\\
&=&\sum_{j=3}^{J-1}C\cdot\min\left\{1,\frac{M^j}{|\xi|}\right\}\cdot L_jA^{-\frac j2}+\sum_{j=J}^\infty C\cdot\min\left\{1,\frac{M^j}{|\xi|}\right\}\cdot L_jA^{-\frac j2}\\
&\le&C\sum_{j=3}^{J-1}\frac{M^j}{|\xi|}\cdot M^{\frac{j\ve}{2}}A^{-\frac j2}+C\sum_{j=J}^\infty M^{\frac{j\ve}{2}}A^{-\frac j2}\\
&\le&C|\xi|^{-1}\sum_{j=3}^{J-1}M^{\left(1-\frac{\delta-\ve}{2}\right)j}+C\sum_{j=J}^\infty M^{-\frac{(\delta-\ve)j}{2}}\\
&\le&C|\xi|^{-1}\cdot\frac{M^{\left(1-\frac{\delta-\ve}{2}\right)J}}{M^{1-\frac{(\delta-\ve)}{2}}-1}+C\cdot\frac{M^{-\frac{(\delta-\ve)J}{2}}}{1-M^{-\frac{(\delta-\ve)}{2}}}\\
&\le&C\cdot\frac{M^{-\frac{(\delta-\ve)J}{2}}}{M^{\frac12}-1}+C\cdot\frac{M^{-
\frac{(\delta-\ve)J}{2}}}{1-M^{-\frac{\delta-\ve}{2}}}\\
&\le&\frac{C}{1-M^{-\frac\delta4}}\cdot|\xi|^{-\frac{\delta-\ve}{2}}\\
&\le&C|\xi|^{-\frac{\delta-\ve}{2}},
\end{eqnarray*}
for some absolute constant $C>0$, provided that $M^{-\frac\delta4}<\frac12$. This can be guaranteed by
$$M\ge 2^\frac4\delta.$$
In this case, if $|\xi|\ge M^4$, then
\begin{eqnarray*}
|\Fc\nu(\xi)|&\le&\left|\Fc\nu_2(\xi)\right|+\sum_{j=3}^\infty\left|\Fc\nu_j(\xi)-\Fc\nu_{j-1}(\xi)\right|\\
&\le&CM^2|\xi|^{-1}+C|\xi|^{-\frac{\delta-\ve}{2}}\\
&\le&C|\xi|^{-\frac12}+C|\xi|^{-\frac{\delta-\ve}{2}}\\
&\le&C|\xi|^{-\frac{\delta-\ve}{2}},
\end{eqnarray*}
because $0<\delta<1$.

\begin{rmk}[Fourier decay estimates of $\nu_j$]
Notice that for each $j\ge3$,
$$\left|\Fc\nu_j(\xi)\right|\le\left|\Fc\nu_2(\xi)\right|+\sum_{k=3}^j\left|\Fc\nu_k(\xi)-\Fc\nu_{k-1}(\xi)\right|.$$
The Fourier decay estimate of the Cantor measure $\nu$ in Theorem \ref{thm:FD} also applies to the Cantor measure $\nu_j$ at each iteration $j\in\N$. 

Recall that the density function $\rho_j$ of the Cantor measure $\nu_j$ is given by 
$$\rho_j(x)=\frac{M^j}{A^j}\Char_{\Cc_j}(x)\quad\text{for }x\in\R.$$
It then follows that
$$\Fc\Char_{\Cc_j}(\xi)=\frac{A^j}{M^j}\Fc\nu_j(\xi)\quad\text{for all }\xi\in\R.$$
Hence, with the same conditions as Theorem \ref{thm:FD},
\begin{equation}\label{eq:FDj}
\left|\Fc\Char_{\Cc_j}(\xi)\right|\le C\frac{A^j}{M^j}|\xi|^{-\frac{\delta-\ve}{2}}\quad\text{for all }j\in\N\text{ and }|\xi|\ge M^4.
\end{equation}
\end{rmk}

\section{From Fourier decay to the FUP}\label{sec:FUP}
In this section, we prove the FUP in Theorems \ref{thm:FUPL} and \ref{thm:FUPC}, following the approach to the FUP by the Fourier decay suggested by Dyatlov \cite[Section V]{Dy}. See also Bourgain-Dyatlov \cite[Section 4]{BD1}, where they used the (generalized) Fourier decay estimate of the Patterson-Sullivan measures to establish an FUP for the related fractal sets (that is, the limit sets of Schottky groups). 

We first prove the FUP in Theorems \ref{thm:FUPC} for the random Cantor measures $\nu$:
$$\|T\|_{L^2_\nu(\R)\to L^2_\nu(\R)}\le Ch^{\frac\delta4-\ve},$$
where
$$Tu(\xi)=\int_{\R}e^{-\frac{ix\cdot \xi}{h}}u(x)\,d\nu(x)\quad\text{for }u\in L^2_\nu(\R).$$
Here, the Cantor set $\Cc$ is chosen from $\Gb\subset\Ab^\infty$ in Theorem \ref{thm:FD} so that the corresponding Cantor measure $\nu$ satisfies the Fourier decay estimate:
$$|\Fc\nu(\xi)|\le C|\xi|^{-\frac{\delta-\ve}{2}}\quad\text{for all }|\xi|\ge M^4.$$
Notice that
$$\|T\|_{L^2_\nu(\R)\to L^2_\nu(\R)}^2=\left\|T^\star T\right\|_{L^2_\nu(\R)\to L^2_\nu(\R)},$$
where $T^\star T$ is an integral operator with kernel
$$K_\nu(\xi,\eta)=\int_\R e^{\frac{i(\xi-\eta)x}{h}}\,d\nu(x)=\sqrt{2\pi}\cdot\Fc\nu\left(\frac{\eta-\xi}{h}\right).$$
The Fourier decay estimate of $\nu$ in Theorem \ref{thm:FD} implies that if $|\frac{\eta-\xi}{h}|\ge M^4$, then
$$\left|K_\nu(\xi,\eta)\right|=\sqrt{2\pi}\cdot\left|\Fc\nu\left(\frac{\eta-\xi}{h}\right)\right|\le C\left|\frac{\eta-\xi}{h}\right|^{-\frac{\delta-\ve}{2}}=Ch^{\frac{\delta-\ve}{2}}|\eta-\xi|^{-\frac{\delta-\ve}{2}}.$$
Here, $C>0$ is an absolute constant. Whereas if $|\frac{\eta-\xi}{h}|\le M^4$, then the trivial estimate holds:
$$\left|K_\nu(\xi,\eta)\right|=\sqrt{2\pi}\cdot\left|\Fc\nu\left(\frac{\eta-\xi}{h}\right)\right|\le C.$$
By Schur's test,
\begin{eqnarray*}
&&\left\|T^\star T\right\|_{L^2_\nu(\R)\to L^2_\nu(\R)}\\
&\le&\sqrt{\sup_{\xi\in\R}\int_{\R}\left|K_\nu(\xi,\eta)\right|\,d\nu(\eta)\cdot\sup_{\eta\in\R}\int_{\R}\left|K_\nu(\xi,\eta)\right|\,d\nu(\xi)}\\
&=&\sup_{\xi\in\R}\int_\R\left|K_\nu(\xi,\eta)\right|\,d\nu(\eta).
\end{eqnarray*}
Since $\nu$ is supported on $[0,1]$, divide the integral to the ones on dyadic intervals of the form
$$2^{-j}\le|\eta-\xi|\le2^{-j+1}\quad\text{with }j=1,...,J,$$
where $J$ is the largest integer such that $2^{-J}\ge M^4h$. Then
$$\left\lfloor\frac{\left|\log\left(M^4h\right)\right|}{\log2}\right\rfloor\le J\le\left\lfloor\frac{\left|\log\left(M^4h\right)\right|}{\log2}\right\rfloor+1.$$
For each $I\subset\R$, $\nu(I)\le C|I|^\delta$. Hence,
\begin{eqnarray*}
&&\int_\R\left|K_\nu(\xi,\eta)\right|\,d\nu(\eta)\\
&\le&\int_{|\eta-\xi|\le M^4h}\left|K_\nu(\xi,\eta)\right|\,d\nu(\eta)+\sum_{j=1}^J\int_{2^{-j}\le|\eta-\xi|\le2^{-j+1}}\left|K_\nu(\xi,\eta)\right|\,d\nu(\eta)\\
&\le&C\left[\nu\left(\left\{\eta:|\eta-\xi|\le M^4h\right\}\right)+\sum_{j=1}^Jh^{\frac{\delta-\ve}{2}}\left(2^{-j}\right)^{-\frac{\delta-\ve}{2}}\cdot\nu\left(\left\{\eta:2^{-j}\le|\eta-\xi|\le2^{-j+1}\right\}\right)\right]\\
&\le&C\left[\left(M^4h\right)^\delta+\sum_{j=1}^Jh^{\frac{\delta-\ve}{2}}\left(2^{-j}\right)^{-\frac{\delta-\ve}{2}}\cdot\left(2^{-j}\right)^\delta\right]\\
&\le&Ch^{\frac{\delta-\ve}{2}}|\log h|\\
&\le&Ch^{\frac\delta2-2\ve},
\end{eqnarray*}
provided that $h\le M^{-8}$ (so $(M^4h)^\delta\le h^\frac\delta2\le h^{\frac{\delta-\ve}{2}}$).

From the Fourier decay of the Cantor measure $\nu_j$ in \eqref{eq:FDj}, we similarly obtain the FUP 
\begin{equation}\label{eq:FUPj}
\left\|\int_{\R}e^{-\frac{ix\cdot \xi}{h}}u(x)\,d\nu_j(x)\right\|_{L^2_{\nu_j}(\R)}\le Ch^{\frac\delta4-\ve}\|u\|_{L^2_{\nu_j}(\R)}.
\end{equation}
Set
$$J=\left\lfloor\frac{|\log h|}{\log M}\right\rfloor.$$
Then $h\approx M^{-J}$, and since $A=M^\delta$, the density function of $\nu_J$ is given by
$$\rho_J(x)=\frac{M^J}{A^J}\Char_{\Cc_J}(x)\approx h^{\delta-1}\Char_{\Cc_J}(x).$$
To prove Theorem \ref{thm:FUPL}, observe that
$$\left\|\Char_{\Cc(h)}\Fc_h\Char_{\Cc(h)}\right\|_{L^2(\R)\to L^2(\R)}\lesssim\left\|\Char_{\Cc_J}\Fc_h\Char_{\Cc_J}\right\|_{L^2(\R)\to L^2(\R)}.$$
To estimate the right-hand side, assume that $u$ is supported in $\Cc_J$. To apply \eqref{eq:FUPj}, we deduce
\begin{eqnarray*}
\left\|\int_{\R}e^{-\frac{ix\cdot \xi}{h}}u(x)\,d\nu_J(x)\right\|_{L^2_{\nu_J}(\R)}&=&h^{\delta-1}\left\|\int_{\Cc_J}e^{-\frac{ix\cdot \xi}{h}}u(x)\,dx\right\|_{L^2_{\nu_j}(\R)}\\
&\approx&h^{\frac32(\delta-1)}\left\|\int_{\Cc_J}e^{-\frac{ix\cdot \xi}{h}}u(x)\,dx\right\|_{L^2(\Cc_J)}\\
&\approx&h^{\frac32(\delta-1)}\cdot\sqrt{2\pi h}\left\|\Char_{\Cc_J}\Fc_hu\right\|_{L^2(\R)},
\end{eqnarray*}
and
$$\|u\|_{L^2_{\nu_J}(\R)}\approx h^{\frac12(\delta-1)}\|u\|_{L^2(\R)}.$$
Hence, \eqref{eq:FUPj} implies
$$h^{\frac32(\delta-1)+\frac12}\left\|\Char_{\Cc_J}\Fc_hu\right\|_{L^2(\R)}\le Ch^{\frac\delta4-\ve}\cdot h^{\frac12(\delta-1)}\|u\|_{L^2(\R)}.$$
That is,
$$\left\|\Char_{\Cc_J}\Fc_hu\right\|_{L^2(\R)}\le Ch^{\frac12-\frac34\delta-\ve}\|u\|_{L^2(\R)}.$$

\end{document}